\documentclass[a4paper, 12pt]{amsart}

\usepackage[a4paper]{geometry}
\usepackage{amssymb}
\usepackage{amsthm}
\usepackage{amsmath}
\usepackage{graphicx}
\usepackage{hyperref}
\usepackage{multirow}
\usepackage{array}
\usepackage[colorinlistoftodos]{todonotes}
\usepackage[english]{babel}
\usepackage{lmodern}
\usepackage{comment}
\usepackage{bbold}
\usepackage{mathtools}
\usepackage{enumerate}
\usepackage{cite}
\usepackage{tikz-cd}

\newcommand{\genlegendre}[4]{%
	\genfrac{(}{)}{}{#1}{#3}{#4}%
	\if\relax\detokenize{#2}\relax\else_{\!#2}\fi
}

\allowdisplaybreaks

\usepackage[utf8]{inputenc}
\usepackage[T1]{fontenc}

\newtheorem{theorem}{Theorem}[section]
\newtheorem{lemma}[theorem]{Lemma}

\theoremstyle{definition}

\newtheorem{proposition}[theorem]{Proposition}

\theoremstyle{remark}
\newtheorem{remark}[theorem]{Remark}
\newcommand{\longcomment}[1]{}

\DeclareMathOperator{\im}{Im}

\newcommand{\GalQ}{\operatorname{Gal}(\overline{\mathbb{Q}}/\mathbb{Q})}
\newcommand{\GalK}{\operatorname{Gal}(\overline{K}/K)}
\newcommand{\GalKv}{\operatorname{Gal}(\overline{K_v}/K_v)}

\newcommand{\whichbold}[1]{\mathbb{#1}} 

\renewcommand{\div}{\mathrm{div}}
\newcommand{\Z}{\whichbold{Z}}

\newcommand{\Q}{\whichbold{Q}}

\newcommand{\C}{\whichbold{C}}
\newcommand{\Qp}{\whichbold{Q}_{p}}

\numberwithin{equation}{section}
\author{Lukas Novak}
\address{Department of Mathematics\\ 
	University of Zagreb\\
	 Bijeni\v{c}ka cesta 30\\
	  10000 Zagreb\\
	  Croatia}
\email{lukas.novak@math.hr}

\title[Twists arising from torsion points]{Twists arising from torsion points}

\makeatletter
\renewcommand{\tocsection}[3]{%
	\indentlabel{\@ifnotempty{#2}{\bfseries\ignorespaces#1 #2\quad}}\bfseries#3}
\renewcommand{\tocsubsection}[3]{%
	\indentlabel{\@ifnotempty{#2}{\ignorespaces#1 #2\quad}}#3}

\AtBeginDocument{%
	\expandafter\renewcommand\csname r@tocindent0\endcsname{0pt}
}
\def\l@subsection{\@tocline{2}{0pt}{2.5pc}{5pc}{}}
\makeatother

\begin{document}

\begin{abstract}
    Let $p$ be a prime number, $K$ a number field that contains the $p$-th root of unity $\zeta_p$, $d$ a $p$-power-free integer and $L=K(\sqrt[p]{d})$. Let $E/K$ be an elliptic curve with full $p$-torsion and $S,T \in E(K)[p]$ be the generators. 

    Define the cocycle $\xi_d : \operatorname{Gal}(\overline{K}/K) \to E$ by
    \[
        \xi_d (\sigma)=
        \begin{cases}
            O, & \text{if } \sigma(\sqrt[p]{d})=\sqrt[p]{d}, \\
            kS, & \text{if } \sigma(\sqrt[p]{d})=\zeta_p^k\sqrt[p]{d},
        \end{cases}
    \]
    and denote by $H_S^d$ the twist of $E$ corresponding to the cocycle $\xi_d$.

    In this paper we construct generators $z$ and $w$ of the function field $K(H_S^d)$ and give a model of the twist 
    \[
        H_S^d\,:\, \alpha_{1}z^p+\alpha_2z^{p-2}w+\dotso+\alpha_{\frac{p+1}{2}}zw^{\frac{p-1}{2}}+\beta w^p+\gamma=0.
    \]
    We also obtain that the twist $H_S^d$ is everywhere locally
    solvable only for finitely many integers~$d$.
\end{abstract}

\maketitle

\tableofcontents

\section{Introduction}\label{sec:1}

Let $C$ be a smooth projective curve over a number field $K$. A \emph{twist} of $C/K$ is a curve $C'/K$ that is isomorphic to $C$ over $\overline{K}$. Furthermore,  we treat two twists as equivalent if they are isomorphic over $K$. The set of twists of $C/K$, modulo $K$-isomorphism, is denoted by $\operatorname{Twist}(C/K)$.

It turns out that the set $\operatorname{Twist}(C/K)$ can be identified with a certain cohomology set.
Namely, there is a one-to-one correspondence between elements in $\operatorname{Twist}(C/K)$ and the elements of the cohomology set $H^1\left(\GalK,\operatorname{Isom}(C)\right)$ where $\operatorname{Isom}(C)$ is the group of all $\overline{K}$-isomorphisms from $C$ to itself.
The proof of this result can be found for example in \cite{silverman2009arithmetic} (Chapter X.2, Theorem 2.2). Specially, this correspondence allows us to define twists of an elliptic curve $E/K$ via cocycles in $H^1\left(\GalK,\operatorname{Isom}(E)\right)$ and vice versa, where $\operatorname{Isom}(E)$ is the isomorphism group of $E$ viewed as a genus one curve. It is known that
\[
    \operatorname{Isom}(E) \cong \operatorname{Aut}(E,O) \ltimes E(\overline{K})
\]
where $\operatorname{Aut}(E,O)$ is the automorphism group of $E$ viewed as an elliptic curve.

One actively studied area for the twists of a curve is the study of local points, everywhere local solvability of these twist and the existence of rational points on the twists.  
The articles below demonstrate some interesting results and the progress made in this area.

Ozman in~\cite{MR2890545} was investigating the existence of local points on the quadratic twist of the modular curve $X_0(N)$ over $\Q$. She gave necessary and sufficient conditions for these quadratic twists to have $\Qp$-rational points over any completion $\Qp$ of $\Q$.

 Ozman continues to study local points on the twist of the modular curve $X_0(N)/\Q$ by a polyquadratic field $K=\Q(\sqrt{d_1},\dotso, \sqrt{d_k})$ in~\cite{MR3071815}. In her work, Ozman gives an algorithm to produce such twists that are everywhere locally solvable. After that, she produced infinitely many twists violating the Hasse principle (i.e.\ the twists that are everywhere locally solvable but do not have a rational point) and also gives an asymptotic formula for the number of such twists.

 Lorenzo Garc\'ia and  Vullers in\cite{MR4770972} inspect the local and global points of the twists of the Klein quartic (a smooth projective plane curve defined by the equation $X^3Y+Y^3Z+Z^3X=0$). The main result in their work provides families with (conjecturally infinitely many) twists of the Klein quartic that violate the Hasse principle.

Kazalicki in his recent work~\cite{MR4935708} studied the quadratic twists of a genus one quartic $H\,:\, y^2=(x^2-x-3)(x^2+2x-12)$ by the integer $d$ and their $2$-Selmer groups. In his work he partially answers the question for which primes $p$ the quadratic twist of $H$ by $\Q(\sqrt{p})$ has a rational point.

Following Kazalicki's work on studying quadratic twists of the genus one quartic $H$, Novak in~\cite{MR4902719} investigates the quadratic twists of genus one quartic curves defined by an irreducible and monic polynomial of degree four. He finds a general criterion for determining local solvability of this twist and also determines the asymptotic growth of the set of square-free integers for which the associated twist is everywhere locally solvable.

Çiperiani and Ozman in~\cite{ciperiani2015local} have studied twists of an elliptic curve $E/K$ defined in the following way:

For any quadratic extension $F=K(\sqrt{d})$ and any point $S\in E(K)$ they considered the cocycle $\xi_{F,S} \in H^1\left(\GalK, \operatorname{Aut}(E)\right)$ defined by
        \[
            \xi_{F,S} (\sigma)=
            \begin{cases}
                (-1,S), & \text{if } \sigma(\sqrt{d})=-\sqrt{d}, \\
                (1,O), & \text{otherwise } 
            \end{cases}
        \]

where $(\pm 1,S)$ sends a point $X \in E$ to the point $\pm X+S$.
Let $E_{F,S}$ denote the twist of $E$ corresponding to $\xi_{F,S}$. Note that $E_{F,S}$ is in general only a genus one curve and not necessarily an elliptic curve. 

In their work, Çiperiani and Ozman showed that this genus one curve $E_{F,S}$ has a $K$-rational point if and only if the point $S$ is in the image of the global trace map $tr_{F/K}$.

If the elliptic curve $E$ does not have complex multiplication, then the elements of $\operatorname{Isom}(E)$ are of the form $(-1,S)$ and $(1,S)$ for a point $S \in E$. Çiperiani and Ozman studied the twists corresponding to the generic element $(-1, S)$  and in this paper we will study the twists corresponding to the other generic element $(1, S)$ of the isomorphism group $\operatorname{Isom}(E)$. We will work in the following setting.

Let $p$ be a prime number, $K$ a number field that contains the $p$-th root of unity $\zeta_p$, $d$ a $p$-power-free integer (an integer that is not divisible by the $p$-th power of any prime number) and $L=K(\sqrt[p]{d})$.
Let $E/K$ be an elliptic curve with full $p$-torsion and $S,T \in E(K)[p]$ the generators of $E(K)[p]$. We define a cocycle $\xi_d : \GalK \to E$ by
\[
    \xi_d (\sigma)=
    \begin{cases}
        O, & \text{if } \sigma(\sqrt[p]{d})=\sqrt[p]{d}, \\
        kS, & \text{if } \sigma(\sqrt[p]{d})=\zeta_p^k\sqrt[p]{d}.
    \end{cases}
\]
Denote by $H_S^d$ the twist of $E$ corresponding to the cocycle $\xi_d$. We investigate the following questions.

\noindent \textbf{Questions:}
    \begin{itemize}
        \item Find a model for the twist $H_S^d$?

        \item  What can we say about the local solvability of $H_S^d$?
    \end{itemize}

In order to obtain a model for the twist $H_S^d$, we define functions $z,w \in \overline{K}(E)$ as
\begin{align*}
    \div{(z)} &= (T)+(T+S)+\dotso+(T+(p-1)S)-(O)-(S)-\dotso -((p-1)S),\\
    \div{(w)} &= (2T)+(2T+S)+\dotso+(2T+(p-1)S)-(O)-(S)-\dotso -((p-1)S).
\end{align*}
Note that such functions $z$ and $w$ do indeed exist because the degrees of the above divisors are $0$, $\displaystyle{\sum_{k=0}^{p-1}(T+kS) - \sum_{k=0}^{p-1}kS = O}$ and $\displaystyle{\sum_{k=0}^{p-1}(2T+kS) - \sum_{k=0}^{p-1}kS = O}$. Furthermore, the functions $z$ and $w$ are only well defined up to multiplication by a constant. In Section~\ref{sec:3} (see Proposition~\ref{prop_zw_def}) we show that there are such constants so that the functions $z$ and $w$ are in the function field $K(H_S^d)$.

With the introduction of this functions we are able to describe a model for the twist $H_S^d$.

\begin{theorem}\label{p-twist_model}
    The functions $z$, $w \in K(H_S^d)$ generate the function field $K(H_S^d)$ and give a model of the twist 
    \[
        H_S^d\,:\, \alpha_{1}z^p+\alpha_2z^{p-2}w+\dotso+\alpha_{\frac{p+1}{2}}zw^{\frac{p-1}{2}}+\beta w^p+\gamma=0
    \]
    where $\alpha_1$, $\alpha_2$, \ldots, $\alpha_{\frac{p+1}{2}}$, $\beta$ and $\gamma \in K$ are constants.
\end{theorem}

Next, by using the well-known Kummer sequence for the elliptic curve $E$
\begin{equation}\label{eq:kummer_seq}
    \begin{tikzcd}
        0 \arrow[r] & E(K_v)/pE[K_v] \arrow[r, "\delta"] & H^1(K_v,E[p]) \arrow[r] & H^1(K_v,E)[p] \arrow[r] & 0.
    \end{tikzcd}    
\end{equation}
and a precise description of $\im{\delta}$, the image of the connecting homomorphisam in the Kummer sequence, for primes $v$ of good reduction we obtain the following result about the locally solubility of the twist $H_S^d$.

\begin{theorem}\label{thm_Hd_ELS_1}
    
    Let $v \in M_K$ be a prime of good reduction such that $v\mid d$ and $v \nmid p \cdot \infty$. 
    Than $H_S^d(K_v) = \emptyset$.
    Specially, $H_S^d$ is everywhere locally solvable only for finitely many $p$-power-free integers $d$.
\end{theorem}

In Section~\ref{sec:2} we will prove Theorem~\ref{p-twist_model} for the cases $p=2$ and $p=3$ by explicitly calculating the functions $z$ and $w$ and writing down the equation for the twist $H_S^d$. After that, we use the model for $H_S^d$ to prove Theorem~\ref{thm_Hd_ELS_1} for $p=2$.
In Section~\ref{sec:3} we used different approaches to obtain the results of Theorems~\ref{p-twist_model} and \ref{thm_Hd_ELS_1}, due to the difficulty in computing the coefficients for the model of $H_S^d$ in the general case.

\section{Twist arising from \texorpdfstring{$2$}{2}-torsion and \texorpdfstring{$3$}{3}-torsion points}\label{sec:2}

\subsection{Twist arising from 2-torsion points}\label{sec:2.1}

Let $E/\Q$ be an elliptic curve and $T\in E(\Q)[2]$ a rational $2$-torsion point on $E$. For a square-free integer $d$, denote by $H^d_T$ the twist of $E$ that corresponds to the cocycle $\zeta : \GalQ \to E$ defined by 
\[
    \zeta (\sigma)=
    \begin{cases}
        O, & \text{if } \sigma(\sqrt{d})=\sqrt{d}, \\
        T, & \text{if } \sigma(\sqrt{d})=-\sqrt{d}.\\
    \end{cases}
\]

By changing the coordinates we may assume that $T=(0,0)$ and $E \,:\, y^2=x^3+ax^2+bx$.
In this case we can write an equation for $H^d_T$ as in \cite{silverman2009arithmetic} (Chapter X.3, Remark 3.7)
\[
    H^d_T \,:\, dy^2=(a^2-4b)x^4-2adx^2+d^2.
\]

\begin{proposition}\label{2-twist_els}
    Let $E/\Q$ be an elliptic curve and $T\in E(\Q)[2]$ a rational $2$-torsion point on $E$. For a square-free integer $d$, let $H^d_T$ be a twist of $E$ defined as above.
    If the curve $H^d_T$ is ELS, then all prime factors of $d$ divide $2a(a^2-4b)$.
    Specially, $H^d_T$ is ELS only for finitely many square-free integers $d$.
\end{proposition}

\begin{proof}
    We will prove this by showing that if there is a prime factor $p$ of $d$ such that ${p \nmid 2a(a^2-4b)}$, then $H^d_T$ is not ELS. 
    Let $d$ be a square-free integer and $p$ a prime factor of $d$ such that $p \nmid 2a(a^2-4b)$. Suppose that $H^d_T$ is ELS. 
    
    Since $H^d_T$ is ELS there exists a point $(x,y)\in H_T^d(\Qp)$. Let us write $x=p^\alpha x_1$ and $y=p^\beta y_1$ for some integers $\alpha$, $\beta$ and $x_1$, $y_1 \in \Z_p^\times$.

    Now we have that 
    \begin{align*}
        v_p(dy^2)&= v_p(dp^{2\beta}y_1)=2\beta+1,\\
        v_p((a^2-4b)x^4-2adx^2+d^2)&=v_p((a^2-4b)p^{4\alpha}x_1^4-2adp^{2\alpha}x_1^2+d^2)=
        \begin{cases}
            2, & \text{if } \alpha \geq 1, \\
            0, & \text{if } \alpha =0, \\
            4\alpha , & \text{if } \alpha < 0.
        \end{cases}
    \end{align*}

    Since $v_p(dy^2)$ is always odd and $v_p((a^2-4b)x^4-2adx^2+d^2)$ is always even, we get a contradiction. Thus, $H^d_T$ is not ELS for such integer $d$.
\end{proof}

\subsection{Twist arising from 3-torsion points}\label{sec:2.2}

Let $K$ be a number field that contains the $3$rd root of unity $\zeta_3$, $d$ a cube-free integer and $L=K(\sqrt[3]{d})$.

Let $E/K$ be an elliptic curve with full $3$-torsion  and $S,T \in E(K)[3]$ the generators of $E(K)[3]$. 
Using Derickx and Sutherland's parametrization for an elliptic curve with points $S$ and $T$ of order $3$ (see~\cite{Derickx_2017}) and by changing the coordinates we may assume that
\[
    E\, : \, y^2+((2+\zeta_3)a+1-\zeta_3)xy+((1+\zeta_3)a^2-\zeta_3 a)y=x^3,
\]
$S=(0,0)$ and $T=(-a,a)$ for some $a \in K$.

We define the cocycle $\xi_d : \GalK \to E$ by
\[
    \xi_d (\sigma)=
    \begin{cases}
        O, & \text{if } \sigma(\sqrt[3]{d})=\sqrt[3]{d}, \\
        S, & \text{if } \sigma(\sqrt[3]{d})=\zeta_3\sqrt[3]{d},\\
        2S, & \text{if } \sigma(\sqrt{d})=\zeta_3^2\sqrt[3]{d}.\\
    \end{cases}
\]
Denote by $H_S^d$ the twist of $E$ corresponding to the cocycle $\xi_d$.

Next, in order to obtain an equation for the twist $H_S^d$ we define functions $z,w \in K(H_S^d)$ as
\begin{align*}
    \div{(z)} &= (T)+(T+S)+(T+2S)-(O)-(S)-(2S),\\
    \div{(w)} &= (2T)+(2T+S)+(2T+2S)-(O)-(S)-(2S).
\end{align*}
Using MAGMA one can easily verify that the functions 
\begin{align*}
    z = & \frac{((1+\zeta_3)a^2+(-2\zeta_3-1)a+\zeta_3)y}{\sqrt[3]{d}x} +\frac{(1+\zeta_3)a^4+(-3\zeta_3-1)a^3+(2\zeta_3-1)a^2+a}{\sqrt[3]{d}x} \\
        &+\frac{(1+\zeta_3)a^3-2\zeta_3 a^2-2a+\zeta_3+1}{\sqrt[3]{d}},
\end{align*}

\begin{align*}
    w = &\frac{((-1-\zeta_3)a^2+(1+2\zeta_3)a-\zeta_3)y}{\sqrt[3]{d^2}x}+\frac{a^4+(-2\zeta_3-3)a^3+(3\zeta_3+2)a^2-\zeta_3 a}{\sqrt[3]{d^2}x}\\
    &+\frac{-\zeta_3 a^3-2a^2+ (2+2\zeta_3)a-\zeta_3}{\sqrt[3]{d^2}}
\end{align*}
satisfy the above divisor conditions and that they are in the function field $K(H_S^d)$. 

With this we are now ready to describe a model for the twist $H_S^d$.

\begin{proposition}\label{3-twist_model}
    With the above notation we have that functions $z$ and $w$ generate the function filed $K(H_S^d)$ and satisfy the equation
    \[
        H_S^d\,:\, dz^3+3d\alpha zw+d^2w^3+\beta=0
    \]
    where $\alpha = a^3+2\zeta_3^2 a^2 +2\zeta_3 a +1$ and 
    $\beta=-a^9+3(3+\zeta_3)a^8-3(\zeta_3+11)a^7+(63\zeta_3+64)a^6-3(35\zeta_3+23)a^5+3(35\zeta_3+12)a^4+(1-63\zeta_3)a^3+3(7\zeta_3-4)a^2+3(2-\zeta_3)a-1$.
\end{proposition}

\begin{proof}
    Using MAGMA one can easily check that for functions $z$ and $w$ the equation $dz^3+3d\alpha zw+d^2w^3+\beta=0$ holds. 

    Let $C \,:\, dz^3+3d\alpha zw+d^2w^3+\beta=0$ be the curve given with the previous equation. We will show that the curve $C$ is isomorphic to the elliptic curve $E$ and that the curve $C$, as a twist of $E$, corresponds to the cocycle $\xi_d$.

    We have the map $\varphi : E \to C$ defined by $\varphi(x,y)=(z,w)$. It can be checked that the map $\varphi$ is indeed a isomorphisam between the curves $E$ and $C$ simply by calculating the inverse $\phi=\varphi^{-1}$. The inverse $\phi : C \to E$ is given by $\phi(z,w)=(X,Y)$ where
    \begin{align*}
        X &= \frac{(\zeta_3+2)a^4+(-5\zeta_3-4)a^3+(5\zeta_3+1)a^2+(1-\zeta_3)a}{\sqrt[3]{d}z+\sqrt[3]{d^2}w-a^3+(2\zeta_3+2)a^2-2\zeta_3a-1},\\
        Y &= \frac{(\zeta_3+2)a^4+(-5\zeta_3-4)a^3+(5\zeta_3+1)a^2+(1-\zeta_3)a}{(1+\zeta_3)a^2+(-2\zeta_3-1)a+\zeta_3}\cdot \frac{\sqrt[3]{d}z+(-\zeta_3-1)a^3+2\zeta_3a^2+2a+(-\zeta_3-1)}{\sqrt[3]{d}z+\sqrt[3]{d^2}w-a^3+(2\zeta_3+2)a^2-2\zeta_3a-1} \\
         & + \frac{(-\zeta_3-1)a^4+(3\zeta_3+1)a^3+(-2\zeta_3+1)a^2-a}{(1+\zeta_3)a^2+(-2\zeta_3-1)a+\zeta_3}.
    \end{align*}

    Since $C$ is isomorphic to $E$ over $L$, we have that the curve $C$ is a twist of the elliptic curve $E$. Let $\xi_C : \GalK \to E$ be the cocycle corresponding to the twist $C$. Than we have that $\xi_C(\sigma) = \phi^\sigma \circ \phi^{-1}$ for all $\sigma \in \GalK$.

    Let us now take a $\sigma \in \GalK$ such that $\sigma(\sqrt[3]{d})= \zeta_3 \sqrt[3]{d}$. Using MAGMA we then calculate
    \[
        \phi^\sigma \circ \phi^{-1} (x,y) = \phi^\sigma(z,w) = \left(\frac{(-\zeta_3-1)a^2+\zeta_3a}{x^2}y, \frac{-\zeta_3a^4-2a^3+(1+\zeta_3)a^2}{x^3}y\right)=(x,y)+S.
    \]
    From the above calculations we have that $\xi_C(\sigma)=\xi_d(\sigma)$ for all such $\sigma \in \GalK$.

    By similar calculations we also get that $\xi_C(\sigma)=\xi_d(\sigma)$ for all the others $\sigma \in \GalK$. Therefore $\xi_C=\xi_d$, from which follows that the curve $C$ is really our twist $H_S^d$.  
\end{proof}

Note that from Theorem~\ref{thm_Hd_ELS_1} for $p=3$ we get a similar result for the twist $H_S^d$ as earlier for the twists arising form $2$-torsion points. Specially, $H_S^d$ is ELS for only finitely many cube-free integers $d$.

\section{Twists arising from \texorpdfstring{$p$}{p}-torsion points}\label{sec:3}

Let $p$ be a prime number, $K$ a number field that contains the $p$-th root of unity $\zeta_p$, $d$ a $p$-power-free integer and $L=K(\sqrt[p]{d})$. 

Let $E/K$ be an elliptic curve with full $p$-torsion  and $S,T \in E(K)[p]$ the generators of $E(K)[p]$. 

We define the cocycle $\xi_d : \GalK \to E$ by
\[
    \xi_d (\sigma)=
    \begin{cases}
        O, & \text{if } \sigma(\sqrt[p]{d})=\sqrt[p]{d}, \\
        kS, & \text{if } \sigma(\sqrt[p]{d})=\zeta_p^k\sqrt[p]{d}.
    \end{cases}
\]
Denote by $H_S^d$ the twist of $E$ corresponding to the cocycle $\xi_d$. 

We also define the \textit{twisted action} of $\GalK$ by the cocycle $\xi_d$. The twisted action of $\sigma \in \GalK$ by $\xi_d$ acts on a function $f \in \overline{K}(E)$ as 
\[
    f^{\sigma^{*}}=f^\sigma \xi_d(\sigma).
\]
It is known that that the function field $K(H_S^d)$ is the fixed field of $\overline{K}(E)$ by the twisted action of $\GalK$ by the cocycle $\xi_d$.

In order to obtain an equation for the twist $H_S^d$, we define functions $z,w \in \overline{K}(E)$ as
\begin{align*}
    \div{(z)} &= (T)+(T+S)+\dotso+(T+(p-1)S)-(O)-(S)-\dotso -((p-1)S),\\
    \div{(w)} &= (2T)+(2T+S)+\dotso+(2T+(p-1)S)-(O)-(S)-\dotso -((p-1)S).
\end{align*}
Note that with this the functions $z$ and $w$ are only well defined up to multiplication by a constant. The next Proposition shows that we can choose this constants in such a way that the functions $z$ and $w$ end up in the function field of the twist $H_S^d$.

\begin{proposition}\label{prop_zw_def}
    There exists functions $z$ and $w$ in the function field $K(H_S^d)$ with the above divisors.
\end{proposition}

\begin{proof}
    Since the divisor $(T)+(T+S)+\dotso+(T+(p-1)S)-(O)-(S)-\dotso -((p-1)S)$ has  degree $0$, its sum of points is $\displaystyle{\sum_{k=0}^{p-1}(T+kS) - \sum_{k=0}^{p-1}kS = O}$ and all the points in the divisor are $K$-rational, we know that there exists a function $f \in K(E)$ such that
    \[
        \div{(f)} = (T)+(T+S)+\dotso+(T+(p-1)S)-(O)-(S)-\dotso -((p-1)S).
    \]
    As functions $z$ and $f$ have the same divisor, it follows that $z=\lambda f$ for some constant $\lambda \in \overline{K}$.

    On the other hand, we also have that the divisor of the function $z\circ \tau_S$, where $\tau_S$ is the translation by the point $S$, is
    \[
        \div{(z\circ \tau_S)} = (T-S)+(T)+\dotso+(T+(p-2)S)-(-S)-(O)-\dotso -((p-2)S) = \div{(z)}
    \]
    Therefore, we have that there is a constant $c \in \overline{K}$ such that $z(P+S) = c z(P)$ for all points $P\in E$. 
    Since $S$ is a $p$-torsion point, we have $z(P) = z(P+pS)=c^p z(P)$ for all points $P \in E$. From there it follows that $c$ is a $p$-th root of unity. Thus, we have that $c=\zeta_p^k$ for some integer~$k$.

    The function $z$ will be in the function field $K(H_S^d)$ if and only if it is fixed by the twisted action action of $\GalK$ by the cocycle $\xi_d$. Let us take a $\sigma \in \GalK$ such that $\sigma(\sqrt[p]{d}) = \zeta_p \sqrt[p]{d}$. We want to choose the constant $\lambda$ such that $z^{\sigma^*}=z$. 
    
    We have that
    \[
        z^{\sigma^*}(P) = z^\sigma(P+S) = (\zeta_p^k z(P))^\sigma= \zeta_p^k \sigma(\lambda) f(P)
    \]
    and $z(P) = \lambda f(P)$ for all points $P \in E$. So, we have to choose the constant $\lambda$ such that $\zeta_p^k \sigma(\lambda) = \lambda$ holds. By choosing $\lambda = \dfrac{1}{\sqrt[p]{d^k}}$ one can easily check that the last equation is satisfied. With such previously chosen constant $\lambda$ one can also easily check that $z^{\sigma^*}=z$ for all the others $\sigma \in \GalK$ and therefore we have that $z \in K(H_S^d)$. 
    
    Using the same arguments as for the function $z$ we also get that there exists such a function $w \in K(H_S^d)$.
\end{proof}

\begin{proposition}\label{p-twist_dep}
    With the notation as above we have that the set $\{z^p, z^{p-2}w,\, \dotso ,\, zw^{\frac{p-1}{2}}, w^p, 1\}$ of functions in $K(H_S^d)$ is $K$-linearly dependent, i.e.\ there exist constants $\alpha_{1}, \alpha_2, \dotso, \alpha_{\frac{p+1}{2}}, \beta, \gamma \in K$ not all zero such that
    \[
        \alpha_{1}z^p+\alpha_2z^{p-2}w+\dotso+\alpha_{\frac{p+1}{2}}zw^{\frac{p-1}{2}}+\beta w^p+\gamma=0.
    \]
\end{proposition}

\begin{proof}
    Let $\varphi$ be the $p$-isogeny of the elliptic curve $E$ with the kernel $\ker{(\varphi)}=\left<S\right>$. Denote by $\tilde{E}$ the image of $\varphi$, $\tilde{T}=\varphi(T)$ and $\tilde{O}=\varphi(O)$. 
    
    Consider the functions $f_1$, $f_2$ and $h \in K(\tilde{E})$ with the following divisors:
    \begin{align*}
        \div{(f_1)} &= p(\tilde{T})-p(\tilde{O}),\\
        \div{(f_2)} &= p(2\tilde{T})-p(\tilde{O}),\\
        \div{(h)} &= (2\tilde{T})+(\tilde{O})-2(\tilde{T}).
    \end{align*}
    Note that the functions $f_1$, $f_2$ and $h$ are  for now only well defined up to multiplication by a constant. For a suitable choice of this constants we can define functions $f_1$, $f_2$ and $h$ such that $\varphi^*(f_1)=z^p$, $\varphi^*(hf_1)=z^{p-2}w$, \ldots, $\varphi^*(h^{\frac{p-1}{2}}f_1)=zw^{\frac{p-1}{2}}$ and $\varphi^*(f_2)=w^p$ holds.
    Thus, in order to show that $z^p, z^{p-2}w,\, \dotso ,\, zw^{\frac{p-1}{2}}, w^p, 1$ are $K$-linearly dependent it is enough to show that $f_1, hf_1,\, \dotso,\ h^{\frac{p-1}{2}}f_1, f_2,1$ are $K$-linearly dependent. 
    
    In order to do so, we will show that there exist a constant $\lambda \in K$ and a polynomial $P \in K[t]$ of degree at most $\frac{p-1}{2}$ such that
    \[
        f_2=\lambda+f_1 P(h).
    \] 
    Set $\lambda=f_2(\tilde{T})$ and consider the function $f=\dfrac{f_2-\lambda}{f_1}$. Since $K$ is and extension of the field $\Q$, we can view the elliptic curve $\tilde{E}$ over $\C$ as a torus $\C/\Lambda$ for some lattice $\Lambda$. According to Proposition 3.4 in Chapter VI in~\cite{silverman2009arithmetic} this allows us to write the functions $f_1$, $f_2$ and $h$ using the Weierstrass $\sigma$-function
    \begin{align*}
        f_1(z) &= C_1\frac{\sigma(z-z_T)^p}{\sigma(z)^p},\\
        f_2(z) &= C_2 \frac{\sigma(z-2z_T)^p}{\sigma(z)^p},\\
        h(z) &= C \frac{\sigma(z-2z_T)\sigma(z)}{\sigma(z-z_T)^2}
    \end{align*}
    where $C_1$, $C_2$ and $C \in K$ are constants and $z_T$ corresponds to point $\tilde{T}$ when viewed in $\C/\Lambda$.
    Then we have that $\lambda=f_2(z_T)=-C_2$ and 
    \[
        f(z)=\frac{f_2(z)-\lambda}{f_1(z)}=\frac{C_2}{C_1}\cdot \frac{\sigma(z-2z_T)^p+\sigma(z)^p}{\sigma(z-z_T)^p}.
    \]
    
    By using the fact that the Weierstrass $\sigma$-function is an odd function we get that the functions $f$ and $h$ are even functions around the point $z_T$ 
    \[
        f(z+z_T)=\frac{C_2}{C_1}\cdot\frac{\sigma(z-z_T)^p+\sigma(z+z_T)^p}{\sigma(z)^p}=f(z_T-z),
    \]
    \[
        h(z+z_T)=C\frac{\sigma(z-z_T)\sigma(z+z_Z)}{\sigma(z)^2}=h(z_T-z).
    \]
    
    Since $h(z)$ is an even function around $z_T$ and has a pole in $\tilde{T}$ of order $2$, its Laurent series around $z_T$ is of the form 
    \[
        h(z)=\sum_{k \geq -1} b_k(z-z_T)^{2k}
    \]
    for some coefficients $b_k \in K$ for all integers $k\geq -1$.
    
    The function $f(z)$ is also even around $z_T$ and has a pole in $\tilde{T}$ of order at most $p-1$ since $f_1$ has a zero in $\tilde{T}$ of order $p$ and $f_2-\lambda$ has a zero in $\tilde{T}$. Thus, the Laurent series of $f(z)$ around $z_T$ is of the form
    \[
        f(z)=\sum_{k \geq -\frac{p-1}{2}} a_k(z-z_T)^{2k}
    \]
    for some coefficients $a_k \in K$ for all integers $k\geq -\frac{p-1}{2}$.
    
    Finally, by using the pole in the point $\tilde{T}$ of order $2$ of the function $h$, we can define a polynomial $P\in K[t]$ of degree at most $\frac{p-1}{2}$ such that it removes the singular part in the Laurent series of $f(z)$, i.e.\ such that we have
    \[
        f(z)-P(h(z))=\sum_{k\geq 0} c_k(z-z_T)^{2k}
    \]
    for some coefficients $c_k \in K$ for all integers $k\geq 0$.

    Notice that the functions $h$ and $f$ have only a pole in the point $\tilde{T}$. Therefore, the only pole of the function $f(z)-P(h(z))$ can eventually only be in the point $\tilde{T}$. But since $f(z)-P(h(z))=\displaystyle{\sum_{k\geq 0} c_k(z-z_T)^{2k}}$, we know that $f(z)-P(h(z))$ does not have a pole in $\tilde{T}$. Thus, the function $f(z)-P(h(z))$ has no poles and hence it must be a constant function. By translating the polynomial $P$ by a suitable constant, we can assume that ${f(z)-P(h(z))=0}$ and from there we get that $f_2=\lambda+f_1 P(h)$ as desired.
\end{proof}

\begin{remark}
    Trying to apply the standard trick using Riemann-Roch Theorem on the Riemann-Roch spaces $\mathcal{L}(nD)$ for the divisor 
    \[
        D = (O)+(S)-\dotso +((p-1)S)
    \]
    and positive integers $n$ does not give us the desired equation for the twist $H_S^d$. 
    
    Namely, we have that all the functions of the form $z^\alpha w^\beta$ for nonnegative integers $\alpha$ and $\beta$ such that $\alpha+\beta \leq n$ are all in the Riemann-Roch space $\mathcal{L}(nD)$. Thus, in order to obtain that the set of function $\{z^\alpha w^\beta\,:\, \alpha + \beta \leq n,\, \alpha, \beta \in \Z_{\geq 0}\}$ is $K$-linearly dependent we need to choose an integer $n$ such that the following inequality holds
    \[
        \binom{n+2}{2} = \#\{z^\alpha w^\beta\,:\, \alpha + \beta \leq n,\, \alpha, \beta \in \Z_{\geq 0}\} > \dim \mathcal{L}(nD) =np.
    \]
    But, the first such integer $n$ for which the above inequality holds is 
    \[
        n = \left\lfloor\dfrac{2p-1+\sqrt{4p^2-12p+1}}{2}\right\rfloor >p.
    \]
\end{remark}

With this we are now finally ready to prove the Theorem~\ref{p-twist_model}, i.e.\ that the functions $z$ and $w$ generate the function field $K(H_S^d)$.

\begin{proof}[Proof of Theorem~\ref{p-twist_model}]
    By Proposition~\ref{p-twist_dep} we already know that there exist constants \newline
    ${\alpha_{1}, \alpha_2, \dotso, \alpha_{\frac{p+1}{2}}, \beta, \gamma \in K}$ not all zero such that $\alpha_{1}z^p+\alpha_2z^{p-2}w+\dotso+\alpha_{\frac{p+1}{2}}zw^{\frac{p-1}{2}}+\beta w^p+\gamma=0$.
    Let us denote by $C$ the curve given with the above equation, i.e.\ 
    \[
        C\,:\, \alpha_{1}z^p+\alpha_2z^{p-2}w+\dotso+\alpha_{\frac{p+1}{2}}zw^{\frac{p-1}{2}}+\beta w^p+\gamma=0.
    \]
    
    First we will show that the curve $C$ is isomorphic to the starting elliptic curve $E$. In order to do so, consider the map $\varphi:E \to C$ defined by $\varphi(x,y)=(z,w)$. We claim that $\deg{(\varphi)}=1$ and from there it follows that $\varphi$ is the desired isomorphisam between $E$ and $C$.

    Indeed, let $(z_0,w_0)\in C$ be a generic point on $C$. The equation $z(P)=z_0$ has exactly $\deg{(z)}=p$ solutions in $E$. Furthermore, since the degree of the curve $C$ is $p$, we know that there are also exactly $p$ points on $C$ with the first coordinate equal to $z_0$. As this is equal to the number of solutions of the equation $z(P)=z_0$, it follows that $\deg{(\varphi)}=1$.

    On the other hand, we also have that $H_S^d$ is isomorphic to $E$ since it is a twist of $E$. Therefore, we have that $H_S^d$ is isomorphic to $C$ and that $\overline{K}(z,w)=\overline{K}(C)=\overline{K}(H_S^d)$. Now by taking the fixed field by the twisted action of the $\GalK$ by the cocycle $\xi_d$ and using that $z$ and $w$ are fixed by this action we finally get $K(z,w)=K(H_S^d)$. This means that the curve $C$ is a model for the twist $H_S^d$ and we can write
    \[
        H_S^d \,:\, \alpha_{1}z^p+\alpha_2z^{p-2}w+\dotso+\alpha_{\frac{p+1}{2}}zw^{\frac{p-1}{2}}+\beta w^p+\gamma=0.
    \]
\end{proof}

Next, we are going to prove Theorem~\ref{thm_Hd_ELS_1} about the local solubility of the twist $H_S^d$. In order to do that we will need the following lemmas.

The first Lemma is due to Cassels and it describes the image of the connecting homomorphisam $\delta$ in the Kummer sequence for primes $v$ of good reduction (see~\cite{MR179169}, Lemma 4.1).
\begin{lemma}[Cassels]\label{lema_cassels}
    Let $E$ be an elliptic curve over the number field $K$, $v \in M_K$ a prime, $p$ a rational prime and $\delta$ the connecting homomorphisam in the Kummer sequence
    \[
        \begin{tikzcd}
            0 \arrow[r] & E(K_v)/pE[K_v] \arrow[r, "\delta"] & H^1(K_v,E[p]) \arrow[r] & H^1(K_v,E)[p] \arrow[r] & 0.
        \end{tikzcd}    
    \]
    If $v$ is a prime of good reduction and $v \nmid p\cdot \infty$ them
    \[
        H_f^1(K_v, E[p]):=\im{\delta}=H^1(K_v^{\operatorname{ur}}/K_v, E[p]).
    \]
    where $K_v^{\operatorname{ur}}$ is the maximal unramified extension of $K_v$.
\end{lemma}

The next Lemma is a small modification of Lemma 2.1 in \cite{ciperiani2015local} and describes a connection to the existence of $K_v$-rational points on $H_S^d$ and the restriction of the cocycle $\xi_d$ to $\GalKv$ for all primes $v$.
\begin{lemma}\label{lema_1}
    Let $v \in M_K$ be a prime and $\xi_d^v$ the restriction of the cocylce $\xi_d$ to $\GalKv$. With the notation as above we have that $H_S^d(K_v) \neq \emptyset$ if and only if the $\{\xi_d^v\}=0$ in $H^1(K_v, E)$.
\end{lemma}
\begin{proof}
    $K_v$-rational points on $H_S^d$ correspond to the points on $E$ that are fixed by the twisted action of the Galois group $\GalKv$. Therefore, we have that
    \begin{align*}
        H_S^d(K_v) &\cong \left\{ R\in E\, :\, R^\sigma = \xi_d^v(\sigma)R \quad \forall \sigma \in \GalKv \right\}\\
        &= \left\{ R\in E\, :\, \xi_d^v(\sigma) = R^\sigma - R \quad \forall \sigma \in \GalKv \right\}.
    \end{align*}
\end{proof}

\begin{proof}[Proof of Theorem~\ref{thm_Hd_ELS_1}]
    Suppose that $H_S^d(K_v)\neq \emptyset$ for such prime $v$. Then by Lemma~\ref{lema_1} we have that $\left\{\xi_d^v\right\}= 0$. From the Kummer sequence 
    \[
        \begin{tikzcd}
            0 \arrow[r] & E(K_v)/pE[K_v] \arrow[r, "\delta"] & H^1(K_v,E[p]) \arrow[r] & H^1(K_v,E)[p] \arrow[r] & 0.
        \end{tikzcd}    
    \]
    it follows then that $\left\{\xi_d^v\right\} \in \im{\delta} = H_f^1(K_v, E[p])$.
    By Lemma~\ref{lema_cassels} we have that $ H_f^1(K_v, E[p])=H^1(K_v^{\operatorname{ur}}/K_v, E[p])$ so we have in fact that $\left\{\xi_d^v\right\} \in H^1(K_v^{\operatorname{ur}}/K_v, E[p])$. But, since $v \mid d$, we have that $\sqrt[p]{d}$ is ramified in $\overline{K_v}$ and because of that $\left\{\xi_d^v\right\}$ is not unramified which leads to a contradiction.
    Therefore, we have that $H_S^d(K_v) = \emptyset$.

    Let us assume now that $H_S^d$ is ELS. From the first part of the proof we know that the only prime divisors of $d$ are the rational primes above the primes of bad reduction of $E$ and eventually the the rational prime $p$. Thus, $d$ has only a finite set of possible rational prime divisors. Sine $d$ is also $p$-power-free it follows that there are only finitely many possible $d$'s for which $H_S^d$ is ELS.  
\end{proof}


\section{Future work}

The paper has left lots of interesting questions for future projects:

\begin{itemize}
    \item Can we classify all the integers $d$ for which the twist $H_S^d$ is ELS in the case $p=2$, $p=3$ or even in the general case?

    \item Can we say something more about the model 
        \[
            H_S^d\,:\, \alpha_{1}z^p+\alpha_2z^{p-2}w+\dotso+\alpha_{\frac{p+1}{2}}zw^{\frac{p-1}{2}}+\beta w^p+\gamma=0 ?
        \]
        What are the singular points of the above model? Is there a connection between the coefficients $\alpha_{1}$, \ldots, $\alpha_{\frac{p+1}{2}}$, $\beta$ and $\gamma$ and the starting elliptic curve $E$, or perhaps even a formula on how to compute them?
    \end{itemize}

\section*{Acknowledgements}

The author was supported by the project “Implementation of cutting-edge research and its application as part of the Scientific Center of Excellence for Quantum and Complex Systems, and Representations of Lie Algebras“, Grant No. PK.1.1.10.0004, co-financed by the European Union through the European Regional Development Fund - Competitiveness and Cohesion Programme 2021-2027 and by the the Croatian Science Foundation under the project no.\ IP-2022-10-5008.

The author would like to thank his mentor Matija Kazalicki on many constructive discussions about the problems in this paper and on numerous helpful ideas and approaches on how to prove some of the results.
    
\bibliographystyle{alpha}
\bibliography{bibliography}

\end{document}